\documentclass[12pt]{article}

\usepackage{amssymb,a4}
\usepackage{amsmath,amsfonts,amssymb,amsthm, mathrsfs}

\newcommand{\1}{\mathbf {1}}
\newcommand{\Z}{{\mathbb Z}}
\newcommand{\C}{{\mathbb C}}
\newcommand{\g}{{\mathfrak g}}

\newcommand{\h}{{\mathfrak h}}
\newcommand{\wh}{{\widehat{\mathfrak h}}}

\newcommand{\mraff}{\mathrm{aff}}
\def\<{\langle}
\def\>{\rangle}
\def\a{\alpha}

\newcommand{\la}{\langle}
\newcommand{\ra}{\rangle}

\newtheorem{thm}{Theorem}[section]
\newtheorem{prop}[thm]{Proposition}
\newtheorem{lem}[thm]{Lemma}

\newtheorem{rmk}[thm]{Remark}
\newtheorem{definition}[thm]{Definition}

\begin{document}

\begin{center}
{\Large \bf  Orbifold theory of the affine vertex operator superalgebra $L_{\widehat{osp(1|2)}}(k,0)$}

\end{center}

\begin{center}
{ Cuipo Jiang$^{a}$\footnote{Supported by China NSF grants No.12171312.}
and Qing Wang$^{b}$\footnote{Supported by
China NSF grants No.12071385 and No.12161141001.}\\
$\mbox{}^{a}$ School of Mathematical Sciences, Shanghai Jiao Tong University, Shanghai 200240, China\\
\vspace{.1cm}
$\mbox{}^{b}$ School of Mathematical Sciences, Xiamen University,
Xiamen 361005, China\\
}
\end{center}

\begin{abstract}
This paper is about the orbifold theory of affine vertex operator superalgebras. Among the main results, we classify the irreducible modules and determine the fusion rules for the orbifold of the simple affine vertex operator superalgebra $L_{\widehat{osp(1|2)}}(k,0)$.

\end{abstract}

\section{Introduction}
\def\theequation{1.\arabic{equation}}
\setcounter{equation}{0}

This paper is a continuation in a series of papers on the study of affine vertex operator superalgebras $L_{\hat{\mathfrak{g}}}(k,0)$ and their coset subalgebras $K(\mathfrak{g},k)$ called parafermion vertex operator algebras with $\g$ a simple Lie superalgebra. If $\g$ is a Lie algebra, $L_{\hat{\g}}(k,0)$ is a $C_2$-cofinite and rational vertex operator algebra if and only if $k$ is a positive integer \cite{FZ}, \cite{DL}, \cite{Li1}. If $\g$ is not a Lie algebra, it was claimed in  \cite{GK}  that $L_{\hat{\g}}(k,0)$ is $C_2$-cofinite if and only if $\g$ is the simple Lie superalgebra $osp(1|2n)$ and $k$ is a positive integer, which was proved recently in \cite{AL} and \cite{CL}. Also in \cite{CL} the rationality of the affine vertex operator superalgebra $L_{\hat{\g}}(k,0)$ with $k$ being a positive integer was shown. It was proved in \cite{JW3} that the commutant $K(\mathfrak{g},k)$ of a Heisenberg vertex operator subalgebra in the simple affine vertex operator superalgebra $L_{\hat{\mathfrak{g}}}(k,0)$ is a simple vertex operator algebra, where $\g=osp(1|2n)$, and the generators were also determined therein. The generator result shows that the parafermion vertex operator algebra $K(osp(1|2),k)$ associated to $osp(1|2)$ together with $K(sl_2,2k)$ associated to $sl_2$ are building blocks of $K(osp(1|2n),k)$. The structure and representation theories for $K(sl_2,k)$ were studied in \cite{DLWY}, \cite{DW1}, \cite{DW2}, \cite{DW3}, \cite{Lam}, \cite{ALY1}, \cite{ALY2} etc. Especially, the orbifold theory of $K(sl_2,k)$ were studied in \cite{JW1}, \cite{JW2}. Recently, the
representation theory for $K(osp(1|2),k)$ have been  studied in \cite{CFK}. It is natural to consider the orbifold theory of the affine vertex operator superalgebra $L_{\hat{\mathfrak{g}}}(k,0)$ and parafermion vertex operator algebra $K(\mathfrak{g},k)$ with $\g=osp(1|2)$ and $k$ being a positive integer. We prove in Theorem \ref{auto.} that the automorphism group of the parafermion vertex operator algebra $K(\mathfrak{g},k)$ with $\g=osp(1|2), k\geq 3$ is  of order 2 and generated by $\sigma$ which is determined by $\sigma(h)=-h, \ \sigma(e)=f,\ \sigma(f)=e,\ \sigma(x+y)=-\sqrt{-1}(x+y),\ \sigma(x-y)=\sqrt{-1}(x-y)$, where $\{e,f,h,x,y\}$ is the basis of the Lie superalgebra $osp(1|2)$ with
$[e,f]=h,$ $[h,e]=2e$, $[h,f]=-2f$
$[h,x]=x$, $[f,x]=-y$,
$[h,y]=-y$,  $[e,y]=-x$,
$\{x,x\}=2e,$ $\{x,y\}=h$, $\{y,y\}=-2f.$ We denote the simple vertex operator superalgebra $L_{\widehat{osp(1|2)}}(k,0)$ by $L_{\hat{\g}}(k,0)$ in this paper. We know that $L_{\hat{\g}}(k,0)=L_{\hat{\g}}^{even}(k,0)\oplus L_{\hat{\g}}^{odd}(k,0)$, where $L_{\hat{\g}}^{even}(k,0)$ is its even subalgebra and $L_{\hat{\g}}^{odd}(k,0)$ is its odd part. Then $\sigma$ can be lifted to an automorphism of $L_{\hat{\g}}(k,0)$ and  is an automorphism of order $2$ when it restricts to the even subalgebra $L_{\hat{\g}}^{even}(k,0)$. Since $L_{\hat{\g}}(k,0)^{\sigma}=L_{\hat{\g}}^{even}(k,0)^{\sigma}$, where $L_{\hat{\g}}(k,0)^{\sigma}$ is the fix-point vertex operator subalgebra of $L_{\hat{\g}}(k,0)$ and $L_{\hat{\g}}^{even}(k,0)^{\sigma}$ is the fix-point vertex operator subalgebra of $L_{\hat{\g}}^{even}(k,0)$, to study the orbifold of the vertex operator superalgebra $L_{\hat{\g}}(k,0)$ under $\sigma$ is essentially to study the orbifold vertex operator algebra $L_{\hat{\g}}^{even}(k,0)^{\sigma}$. In this paper, we classify the irreducible modules of $L_{\hat{\g}}(k,0)^{\sigma}$ and determine their fusion rules. We will study irreducible modules and fusion rules for the orbifold parafermion vertex algebra $K(g,k)^{\sigma}$ in subsequent papers.

From \cite{CL} and \cite{Dra}, we know that the commutant of affine vertex operator algebra $L_{\hat{sl_2}}(k,0)$ in $L_{\hat{\g}}(k,0)$ is the minimal Virasoro vertex algebra $L^{Vir}(c_{2k+3,k+2},0)$, and the decomposition of $L_{\hat{\g}}(k,0)$ as an $L_{\hat{sl_2}}(k,0)\otimes L^{Vir}(c_{2k+3,k+2},0)$-module was obtained in \cite{CFK} by a character decomposition of simple highest weight modules of $\widehat{osp(1|2)}$. By applying the classification result of the irreducible modules of the vertex operator algebra $L_{\hat{\g}}^{even}(k,0)$ in \cite{CFK} and the classification result of the irreducible modules of the orbifold vertex operator algebra $L_{\hat{sl_2}}(k,0)^{\sigma}$ we obtained in \cite{JW3}, together with the analysis of the lowest weights in the top level of these irreducible modules, we obtain all the irreducible modules of $L_{\hat{\g}}(k,0)^{\sigma}$ in Theorem \ref{thm:affinetwist'}. There are two kinds of irreducible modules, one is the untwisted type and the other is the twisted type. Furthermore, we determine the contragredient modules of all irreducible $L_{\hat{\g}}(k,0)^{\sigma}$-modules in Theorem \ref{thm:contragredient aff}. This result together with the symmetric property of fusion rules in \cite{FHL} implies that we only need to consider the
 fusion products between  the untwisted type modules, and the fusion products between the untwisted type modules and the twisted type modules. We obtain these fusion rules mainly by using the decomposition of irreducible modules of $L_{\hat{\g}}(k,0)^{\sigma}$ in Theorem \ref{thm:affinetwist'} and the fusion rules of $L_{\hat{sl_2}}(k,0)^{\sigma}$ we obtained in \cite{JW3}.

The paper is organized as follows. In Section 2, we recall some results about the affine vertex operator superalgebra $L_{\hat{\g}}(k,0)$ and its parafermion vertex operator subalgebra $K(osp(1|2),k)$. Then we determine the automorphism group of $K(osp(1|2),k)$ by Theorem \ref{auto.}. In Section 3, we classify the irreducible modules of the orbifold $L_{\hat{\g}}(k,0)^{\sigma}$ of the affine vertex operator superalgebra $L_{\hat{\g}}(k,0)$. In Section 4, we determine the fusion rules for $L_{\hat{\g}}(k,0)^{\sigma}$.

\section{Preliminaries}
\label{Sect:V(k,0)}\def\theequation{2.\arabic{equation}}

In this section, we recall from \cite{JW2}, \cite{JW3} some basic results on the affine vertex operator superalgebra and parafermion vertex
operator algebra of $osp(1|2)$ at level $k$ with $k$ being a positive integer.

Let $\g$ be the finite dimensional simple Lie superalgebra $osp(1|2)$ with a Cartan
subalgebra $\h.$ Let $\{e,f,h,x,y\}$ be the basis of the Lie superalgebra $\g$ with the anti-commutation relations:
$$[e,f]=h,\; [h,e]=2e, \; [h,f]=-2f$$
$$[h,x]=x, \; [e,x]=0, \; [f,x]=-y$$
$$[h,y]=-y, \; [e,y]=-x, \; [f,y]=0$$
$$\{x,x\}=2e, \; \{x,y\}=h, \; \{y,y\}=-2f.$$ Let  $\la ,\ra$ be an invariant even supersymmetric
nondegenerate bilinear form on $\g$ such that $\<\a,\a\>=2$ if
$\alpha$ is a long root in the root system of even, where we have identified $\h$ with $\h^*$
via $\<,\>.$  Let $\hat{\mathfrak g}= \g \otimes \C[t,t^{-1}] \oplus \C K$
be the corresponding affine Lie superalgebra. Let $k$ be a positive
integer and
\begin{equation*}
 V_{\hat{\g}}(k,0) = Ind_{\g \otimes
\C[t]\oplus \C K}^{\hat{\g}}\C
\end{equation*}
the induced $\hat{\g}$-module such that ${\g} \otimes \C[t]$ acts as $0$ and $K$ acts as $k$ on $\1=1$.

We denote by $a(n)$ the operator on $V_{\hat{\g}}(k,0)$ corresponding to the action of
$a \otimes t^n$. Then
$$[a(m), b(n)] = [a,b](m+n) + m \la a,b \ra \delta_{m+n,0}k$$
for $a, b \in \g$ and $m,n\in \Z$.

Let $a(z) = \sum_{n \in \Z} a(n)z^{-n-1}$. Then $V_{\hat{\g}}(k,0)$ is a
vertex operator superalgebra generated by $a(-1)\1$ for $a\in \g$ such
that $Y(a(-1)\1,z) = a(z)$ with the
vacuum vector $\1$ and the Virasoro vector
\begin{align*}
\omega_{\mraff} &= \frac{1}{2(k+\frac{3}{2})} \Big(
\frac{1}{2}h(-1)h(-1)\1 +e(-1)f(-1)\1+f(-1)e(-1)\1
\\&-
\frac{1}{2}x(-1)y(-1)\1+
\frac{1}{2}y(-1)x(-1)\1
\Big)
\end{align*}
of central charge $\frac{2k}{2k+3}$(see \cite{KRW}).

Let $M(k)$ be the vertex operator subalgebra of $V_{\hat{\g}}(k,0)$
generated by $h(-1)\1$ with the Virasoro
element
$$\omega_{\gamma} = \frac{1}{4k}
h(-1)^{2}\1$$
of central charge $1$.


The vertex operator superalgebra $V_{\hat{\g}}(k,0)$ has a unique maximal ideal $\mathcal{J}$, which is generated by a weight $k+1$ vector $e(-1)^{k+1}\1$ \cite{GS}. The quotient algebra $L_{\hat{\g}}(k,0)=V_{\hat{\g}}(k,0)/\mathcal{J}$ is a
simple, rational and $C_2$-cofinite vertex operator superalgebra \cite{AL}, \cite{CL}. Moreover, the image of $M(k)$
in $L(k,0)$ is isomorphic to $M(k)$ and will be
denoted by $M(k)$ again. Set
\begin{equation*}
 K(\g,k)=\{v \in L_{\hat{g}}(k,0)\,|\, h(m)v =0
\text{ for }\; h\in {\mathfrak h},
 m \ge 0\}.
\end{equation*}
Then $K(\g,k)$, which is the space of highest weight vectors
with highest weight $0$ for $\wh$,
is the commutant of $M(k)$ in $L_{\hat{\g}}(k,0)$. We proved in \cite{JW3} that $K(\g,k)$ is a simple vertex operator algebra which is called the parafermion vertex operator algebra and generated by

\begin{equation}\label{eq:w3}
\begin{split}
\omega
=\frac{1}{2k(k+2)}(-h(-1)^{2}{\1}
+2ke(-1)f(-1){\1}-kh(-2){\1}),
\end{split}
\end{equation}

\begin{equation}\label{eq:w3'}
\begin{split}
\bar{\omega}
=-h(-1)^{2}{\1}
+4kx(-1)y(-1){\1}-2kh(-2){\1},
\end{split}
\end{equation}

\begin{equation}\label{eq:W3'}
\begin{split}
W^3 &= k^2 h(-3){\1} + 3 k
h(-2)h(-1){\1} + 2h(-1)^3{\1}
-6k h(-1)e(-1)f(-1){\1}
\\
&+3k^2e(-2)f(-1){\1}
-3k^2e(-1)f(-2){\1},
\end{split}
\end{equation}

\begin{equation}\label{eq:W3'''}
\begin{split}
\bar{W}^3 &= 2k^2 h(-3){\1} + 3 k
h(-2)h(-1){\1} + h(-1)^3{\1}
-6k h(-1)x(-1)y(-1){\1}\\&
+6k^2x(-2)y(-1){\1}-6k^2x(-1)y(-2){\1}.\end{split}
\end{equation}

 Let $L_{\hat{sl_2}}(k,0)$ be the simple affine vertex operator algebra associated to $\hat{sl_2}$ and let $L(k,i)$ for $0\leq i\leq k$ be the irreducible modules for the rational vertex operator algebra $L_{\hat{sl_2}}(k,0)$ with the top level $U^{i}=\bigoplus_{j=0}^{i}\mathbb{C}v^{i,j}$ which is an $(i+1)$-dimensional irreducible module of the simple Lie algebra $\C h(0)\oplus\C e(0)\oplus \C f(0)\cong sl_2$.

Let $\sigma$ be an automorphism of the Lie superalgebra $osp(1|2)$ defined by $\sigma(h)=-h, \ \sigma(e)=f,\ \sigma(f)=e,\ \sigma(x+y)=-\sqrt{-1}(x+y),\ \sigma(x-y)=\sqrt{-1}(x-y)$. $\sigma$ can be lifted to an automorphism $\sigma$ of the vertex operator superalgebra $V_{\hat{\g}}(k,0)$ of order 4 in the following way:
$$\sigma(a_{1}(-n_{1})\cdots a_{s}(-n_{s})\1)=\sigma(a_{1})(-n_{1})\cdots \sigma(a_{s})(-n_{s})\1$$
for $a_{i}\in osp(1|2)$ and $n_{i}>0$. Then $\sigma$ induces an automorphism of $L_{\hat{\g}}(k,0)$ as $\sigma$ preserves the unique maximal ideal $\mathcal{J}$, and the Virasoro element $\omega_{\gamma}$ is invariant under $\sigma$. Thus $\sigma$ induces an automorphism of the parafermion vertex operator algebra $K(\g,k)$. In fact, $\sigma(\omega)=\omega,\ \sigma(\bar{\omega})=\bar{\omega}, \sigma(W^{3})=-W^{3},\ \sigma(\bar{W^{3}})=-\bar{W^{3}}$.

 Now we determine the automorphism group of $K(\g,k)$. Let $L^{Vir}(c_{p,q},0)$ be the minimal Virasoro vertex operator algebra with central charge $c_{p,q}=1-\frac{6(p-q)^{2}}{pq}$, $p,q\in\mathbb{Z}_{\geq 2}, (p,q)=1.$  It is known that \cite{W} $L^{Vir}(c_{p,q},0)$ is rational and its irreducible modules are $\{L^{Vir}(c_{p,q},h_{p,q}^{r,s})|h_{p,q}^{r,s}\in S\}$, where

$$S=\{h_{p,q}^{r,s}=\frac{(sq-rp)^{2}-(p-q)^{2}}{4pq}|1\leq r\leq q-1, 1\leq s\leq p-1\}.$$ We denote these irreducible modules $L^{Vir}(c_{p,q},h_{p,q}^{r,s})$ by  $V_{r,s}$ for $1\leq r\leq q-1, 1\leq s\leq p-1$.

\begin{thm}\label{auto.}  Aut$(K(\g,k))=\mathbb{Z}/2\mathbb{Z}$ if $k\geq3$.
\end{thm}
\begin{proof} From \cite{DLY2}, we know that Aut$(K(sl_2,k))$ is generated by $\sigma$ of order $2$ if $k\geq3$, where $\sigma$ is determined by $\sigma(h)=-h$, $\sigma(e)=f$, $\sigma(f)=e$. We claim that Aut$(K(\g,k))=\mbox{Aut}(K(sl_2,k))$, where $K(sl_2,k)$ is the parafermion vertex operator algebra associated to $\hat{sl_2}$. In fact, we first prove that Aut$(K(\g,k))$ is a subgroup of Aut$(K(sl_2,k))$. Note that from \cite{CL1}, we know that the coset $$Com(L_{\hat{sl_2}}(k,0), L_{\hat{\g}}(k,0))=L^{Vir}(c_{2k+3,k+2},0),$$ where $L^{Vir}(c_{2k+3,k+2},0)$ is the minimal Virasoro vertex operator algebra with central charge $c_{2k+3,k+2}=1-\frac{6(k+1)^{2}}{(2k+3)(k+2)}$. Note that
$V_{\mathbb{Z}\gamma}\subseteq L_{\hat{sl_2}}(k,0)\subseteq L_{\hat{\g}}(k,0)$, where $V_{\mathbb{Z}\gamma}$ is the lattice vertex operator algebra associated with a rank one lattice $\mathbb{Z}\gamma$ and $\langle\gamma,\gamma\rangle=2k$.
From Lemma 4.3 in \cite{DLY2} and the module decomposition of the vertex operator superalgebra $L_{\hat{\g}}(k,0)$ as $L_{\hat{sl_2}}(k,0)$-modules \cite{CFK}, we deduce that
$$K(\g,k)=K(sl_2,k)\otimes L^{Vir}(c_{2k+3,k+2},0)\bigoplus_{j=1}^{k-1}M^{2j,j}\otimes L^{Vir}(c_{2k+3,k+2},h_{2k+3,k+2}^{1,2j+1})$$
as a $K(sl_2,k)$-module, where $M^{i,j}$ for $0\leq i\leq k, 0\leq j\leq k-1$ are irreducible modules of $K(sl_2,k)$.
For any $\sigma\in \mbox{Aut}(K(\g,k))$, notice that the automorphism $\sigma$ preserves the weight space. This implies that
 $$\sigma(K(sl_2,k)\otimes L^{Vir}(c_{2k+3,k+2},0))=K(sl_2,k)\otimes L^{Vir}(c_{2k+3,k+2},0).$$ Let $\omega_{1}, \omega, \omega_{2}$ be the Virasoro elements of $K(\g,k), K(sl_2,k), L^{Vir}(c_{2k+3,k+2},0)$ respectively. We know that the weight $2$ subspace $(K(sl_2,k)\otimes L^{Vir}(c_{2k+3,k+2},0))_2$ is spanned by $\omega$ and $\omega_{2}$. Then together with the fact that Com$(K(sl_2,k),K(\g,k))=L^{Vir}(c_{2k+3,k+2},0)$, we deduce that $\sigma(\omega_{2})=\omega_{2}$. Since Com$(L^{Vir}(c_{2k+3,k+2},0),K(\g,k))=K(sl_2,k)$, it follows that for any $v\in K(sl_2,k)$,
 $\sigma(v)\in K(sl_2,k)$, that is, the restriction of $\sigma$ on $K(sl_2,k)$ is $K(sl_2,k)$. Thus Aut$(K(\g,k))$ is a subgroup of Aut$(K(sl_2,k))$.

 Next we construct an automorphism $\sigma$ of $K(\g,k)$ such that $\sigma^{2}=1$. Let $\sigma(e)=f, \sigma(f)=e, \sigma(h)=-h$, where $e,f,h$ is a $sl_2$-triple. We know that $\sigma$ can be lifted to the automorphism of $K(sl_2,k)$ such that $\sigma(\omega)=\omega, \sigma(W^{3})=W^{3}$, where $\omega$ and $W^{3}$ are the generators of $K(sl_2,k)$ with weight $2$ and $3$ respectively. Notice that
 \begin{eqnarray*}
y(-1)x(-1)\1+x(-1)y(-1)\1=h(-2)\1.
\end{eqnarray*}
Then we have
\begin{eqnarray}
\sigma(y(-1)x(-1)\1+x(-1)y(-1)\1)=-y(-1)x(-1)\1-x(-1)y(-1)\1.\label{eq:5.1}
\end{eqnarray}
Since Com$(K(sl_2,k),K(\g,k))=L^{Vir}(c_{2k+3,k+2},0)$ and $\omega_1=\omega+\omega_2$, it follows that $\sigma(\omega_{2})=\omega_{2}$. Then
we have \begin{eqnarray}
\sigma(x(-1)y(-1)\1-y(-1)x(-1)\1)=x(-1)y(-1)\1-y(-1)x(-1)\1.\label{eq:5.2}
\end{eqnarray}
From (\ref{eq:5.1}) and (\ref{eq:5.2}), we have
\begin{eqnarray*}
\sigma(y(-1)x(-1)\1)=-x(-1)y(-1)\1-y(-1)x(-1)\1, \sigma(x(-1)y(-1)\1)=-y(-1)x(-1)\1.
\end{eqnarray*}
This shows that $\sigma(\omega)=\omega,\sigma(\bar{\omega})=\bar{\omega}$, $\sigma(W^{3})=-W^{3},\sigma(\bar{W^3})=-\bar{W^3}$, where $\omega,\bar{\omega},W^{3},\bar{W^3}$
are generators of $K(\g,k)$ given in (\ref{eq:w3})-(\ref{eq:W3'''}). Thus $\sigma$ lifts to an automorphism of $K(\g,k)$ and $\sigma^{2}=1$. So we have Aut$(K(\g,k))=\mbox{Aut}(K(sl_2,k))=\mathbb{Z}/2\mathbb{Z}$ if $k\geq3$.

\end{proof}

\begin{rmk}  From the proof of Theorem \ref{auto.}, we see that Aut$(K(\g,k))=\mbox{Aut}(K(sl_2,k))$ for any positive integer $k$. Thus the automorphism group Aut$(K(\g,k))=\{1\}$ is trivial for $k=1$ and $k=2$. Therefore, we only need to consider the orbifold of parafermion vertex operator algebra under the automorphism $\sigma$ for $k\geq 3$. So we first need to understand the orbifold affine vertex operator superalgebra $L_{\hat{g}}(k,0)^{\sigma}$, i.e., the fixed-point vertex operator subalgebra of $L_{\hat{g}}(k,0)$ under the automorphism $\sigma$ of order $4$. We will study modules and fusion rules of $K(\g, k)^{\sigma}$ in subsequent papers.
\end{rmk}

\section{Classification of the irreducible modules of $L_{\hat{g}}(k,0)^{\sigma}$
}\label{Sect:affine case}\def\theequation{3.\arabic{equation}}

In this section, we classify all the irreducible modules of the orbifold vertex operator algebra $L_{\hat{g}}(k,0)^{\sigma}$.
 We first recall the definition of weak $g$-twisted modules, $g$-twisted modules and admissible $g$-twisted modules following \cite{DLM3, DLM4}. Let $L_{\hat{g}}(k,0)^{\sigma}$ be the orbifold vertex operator subalgebra of the affine vertex operator superalgebra $L_{\hat{g}}(k,0)$, i.e., the fixed-point subalgebra of $L_{\hat{g}}(k,0)$ under $\sigma$. We then classify the irreducible modules for $L_{\hat{g}}(k,0)^{\sigma}$. Furthermore, we determine the contragredient modules of irreducible $L_{\hat{g}}(k,0)^{\sigma}$-modules.

Let $\left(V,Y,1,\omega\right)$ be a vertex operator algebra (see
\cite{FLM}, \cite{LL}) and $g$ an automorphism of $V$ with finite order $T$.  Let $W\left\{ z\right\} $
denote the space of $W$-valued formal series in arbitrary complex
powers of $z$ for a vector space $W$. Denote the decomposition of $V$ into eigenspaces with respect to the action of $g$ by $$V=\bigoplus_{r\in\Z}V^{r},$$
where $V^{r}=\{v\in V|\ gv=e^{-\frac{2\pi ir}{T}}v\},$ $i=\sqrt{-1}$.

\begin{definition}A \emph{weak $g$-twisted $V$-module} $M$ is
a vector space with a linear map
\[
Y_{M}:V\to\left(\text{End}M\right)\{z\}
\]

\[
v\mapsto Y_{M}\left(v,z\right)=\sum_{n\in\mathbb{Q}}v_{n}z^{-n-1}\ \left(v_{n}\in\mbox{End}M\right)
\]
which satisfies the following conditions for $0\leq r\leq T-1$, $u\in V^{r}\ ,v\in V, w\in M$:

\[ Y_{M}\left(u,z\right)=\sum_{n\in\frac{r}{T}+\mathbb{Z}}u_{n}z^{-n-1}
\]

\[
u_{n}w=0\ {\rm for} \ n\gg0,
\]

\[
Y_{M}\left(\mathbf{1},z\right)=Id_{M},
\]

\[
z_{0}^{-1}\text{\ensuremath{\delta}}\left(\frac{z_{1}-z_{2}}{z_{0}}\right)Y_{M}\left(u,z_{1}\right)
Y_{M}\left(v,z_{2}\right)-z_{0}^{-1}\delta\left(\frac{z_{2}-z_{1}}{-z_{0}}\right)Y_{M}\left(v,z_{2}\right)Y_{M}\left(u,z_{1}\right)
\]

\[
=z_{1}^{-1}\left(\frac{z_{2}+z_{0}}{z_{1}}\right)^{\frac{r}{T}}\delta\left(\frac{z_{2}+z_{0}}{z_{1}}\right)Y_{M}\left(Y\left(u,z_{0}\right)v,z_{2}\right),
\]
 where $\delta\left(z\right)=\sum_{n\in\mathbb{Z}}z^{n}$. \end{definition}
 The following identities are the consequences of the twisted-Jacobi identity \cite{DLM3} (see also \cite{Ab}, \cite{DJ}).
 \begin{eqnarray}[u_{m+\frac{r}{T}},v_{n+\frac{s}{T}}]=\sum_{i=0}^{\infty}\binom{m+\frac{r}{T}}{i} (u_{i}v)_{m+n+\frac{r+s}{T}-i},\label{eq:3.1.}\end{eqnarray}
 \begin{eqnarray}\sum_{i\geq 0}\binom{\frac{r}{T}}{i}(u_{m+i}v)_{n+\frac{r+s}{T}-i}=\sum_{i\geq 0}(-1)^{i}\binom{m}{i}(u_{m+\frac{r}{T}-i}v_{n+\frac{s}{T}+i}-(-1)^{m}v_{m+n+\frac{s}{T}-i}u_{\frac{r}{T}+i}),\label{eq:3.2.}\end{eqnarray}
 where $u\in V^{r}, \ v\in V^{s},\ m,n\in \Z$.

\begin{definition}

A \emph{$g$-twisted $V$-module} is a weak $g$-twisted $V$-module\emph{
}$M$ which carries a $\mathbb{C}$-grading $M=\bigoplus_{\lambda\in\mathbb{C}}M_{\lambda},$
where $M_{\lambda}=\{w\in M|L(0)w=\lambda w\}$ and $L(0)$ is one of the coefficient operators of $Y(\omega,z)=\sum_{n\in\mathbb{Z}}L(n)z^{-n-2}.$
Moreover we require
that $\dim M_{\lambda}$ is finite and for fixed $\lambda,$ $M_{\lambda+\frac{n}{T}}=0$
for all small enough integers $n.$

\end{definition}

\begin{definition}An \emph{admissible $g$-twisted $V$-module} $M=\oplus_{n\in\frac{1}{T}\mathbb{Z}_{+}}M\left(n\right)$
is a $\frac{1}{T}\mathbb{Z}_{+}$-graded weak $g$-twisted module
such that $u_{m}M\left(n\right)\subset M\left(\mbox{wt}u-m-1+n\right)$
for homogeneous $u\in V$ and $m,n\in\frac{1}{T}\mathbb{Z}.$ $ $

\end{definition}

If $g=Id_{V}$, we have the notions of weak, ordinary and admissible
$V$-modules \cite{DLM3}.

\begin{definition}A vertex operator algebra $V$ is called \emph{$g$-rational}
if the admissible $g$-twisted module category is semisimple. \end{definition}

The following lemma about $g$-rational vertex operator algebras is
well known \cite{DLM3}.

\begin{lem} If $V$ is  $g$-rational, then

(1) Any irreducible admissible $g$-twisted $V$-module $M$ is a $g$-twisted $V$-module, and there exists a $\lambda\in\mathbb{C}$
such that $M=\oplus_{n\in\frac{1}{T}\mathbb{Z}_{+}}M_{\lambda+n}$
where $M_{\lambda}\neq0.$ And $\lambda$ is called the conformal weight
of $M;$

(2) There are only finitely many irreducible admissible $g$-twisted
$V$-modules up to isomorphism. \end{lem}



Let $M=\bigoplus_{n\in\frac{1}{T}\mathbb{Z}_{+}}M(n)$ be an admissible $g$-twisted $V$-module, the contragredient module $M^{'}$ is defined as follows: $M'=\bigoplus_{n\in\frac{1}{T}\mathbb{Z}_{+}}M(n)^{*}$, where $M(n)^{*}=\mbox{Hom}_{\mathbb{C}}(M(n),\mathbb{C}).$
The vertex
operator $Y_{M'}(v,z)$ is defined for $v\in V$ via
\begin{eqnarray}\label{contragredient}
\langle Y_{M'}(v,z)f,u\rangle=\langle f,Y_{M}(e^{zL(1)}(-z^{-2})^{L(0)}v,z^{-1})u\rangle,\label{eq:3.0}
\end{eqnarray}
where $\langle f,w\rangle=f(w)$ is the natural paring $M'\times M\to\mathbb{C}.$

\begin{rmk} $(M^{'}, Y_{M^{'}})$ is an admissible $g^{-1}$-twisted $V$-module \cite{FHL}. One can also define the contragredient module $M^{'}$ for a $g$-twisted $V$-module $M$. In this case, $M^{'}$ is a $g^{-1}$-twisted $V$-module. Moreover, $M$ is irreducible if and only if $M^{'}$ is irreducible.
\end{rmk}

  Recall that $L(k,i)$ for $0\leq i\leq k$ be all the irreducible modules for the rational vertex operator algebra $L_{\hat{sl_{2}}}(k,0)$. From \cite{CFK}, we know that the vertex operator superalgebra $L_{\hat{g}}(k,0)=L^{even}_{\hat{g}}(k,0)\oplus L^{odd}_{\hat{g}}(k,0)$, where $L^{even}_{\hat{g}}(k,0)$ is its even subalgebra which is simple and $L^{odd}_{\hat{g}}(k,0)$ is an order two simple current with $$L^{even}_{\hat{g}}(k,0)=\bigoplus_{i=0, \ i\ even}^{k}L(k,i)\otimes V_{i+1,1} $$and
  $$L^{odd}_{\hat{g}}(k,0)=\bigoplus_{i=0, \ i\ odd}^{k}L(k,i)\otimes V_{i+1,1}.$$
From the definition of the automorphism $\sigma$, we see that $L_{\hat{g}}(k,0)^{\sigma}=L^{even}_{\hat{g}}(k,0)^{\sigma}$, and $\sigma$ is the automorphism of vertex operator algebra $L^{even}_{\hat{g}}(k,0)$ with order $2$. Thus the study of orbifold theory of vertex operator superalgebra $L_{\hat{g}}(k,0)$ under $\sigma$ is equivalent to the study of the $\mathbb{Z}_{2}$-orbifold vertex operator algebra $L^{even}_{\hat{g}}(k,0)^{\sigma}$.

 The following theorem gives the classification of the irreducible modules of $L^{even}_{\hat{g}}(k,0)$.

 \begin{thm}\cite{CFK}\label{thm:affine'} The modules
 \begin{eqnarray}\label{decomposition}
 M_{r}^{even}=\bigoplus_{i=0, \ i\ even}^{k}L(k,i)\otimes V_{i+1,r} \  \mbox{and} \ \ M_{r}^{odd}=\bigoplus_{i=0, \ i\ odd}^{k}L(k,i)\otimes V_{i+1,r}
  \end{eqnarray} for $1\leq r\leq 2k+2$ form a complete list of inequivalent simple modules of the vertex operator algebra $L^{even}_{\hat{g}}(k,0)$.

 \end{thm}

Note that $L_{\hat{sl_{2}}}(k,0)^{\sigma}$ is the subalgebra of $L^{even}_{\hat{g}}(k,0)^{\sigma}$. The irreducible modules of $L_{\hat{sl_{2}}}(k,0)^{\sigma}$ are classified in \cite{JW2}.

\begin{thm}\cite{JW2}\label{thm:affine sl2}
There are $4k+4$ inequivalent irreducible modules of $L_{\hat{sl_{2}}}(k,0)^{\sigma}$.
  More precisely, the set
 \begin{eqnarray*}
&& \{ L(k,i)^{+}, L(k,i)^{-}, \overline{L(k,i)}^{+}, \overline{L(k,i)}^{-} \ \mbox{for} \ 0\leq i\leq k\} \\
 \end{eqnarray*} gives all inequivalent irreducible $L_{\hat{sl_{2}}}(k,0)^{\sigma}$-modules.
\end{thm}

\begin{rmk}\label{rmk:affine sl2} With the notations in  Theorem \ref{thm:affine sl2}, we call $L(k,i)^{+}, L(k,i)^{-}$ untwisted type modules and  $\overline{L(k,i)}^{+}, \overline{L(k,i)}^{-} $ twisted type modules respectively.

\end{rmk}
Since $L_{\hat{g}}^{even}(k,0)$ is a rational vertex operator algebra, $L_{\hat{g}}^{even}(k,0)^{\sigma}$ is $C_2$-cofinite and rational \cite{M2}, \cite{CM}, \cite{CKLR}, it follows that $L_{\hat{g}}^{even}(k,0)$ is $\sigma$-rational \cite{DH}. Then from \cite{DLM4}, we have the following result.

\begin{prop}\label{prop:twisted1}
 There are precisely $4k+4$ inequivalent irreducible $\sigma$-twisted modules of $L^{even}_{\hat{g}}(k,0)$ as follows: \begin{eqnarray}\label{decompositiontwist}
 \overline{M_{r}^{even}}=\bigoplus_{i=0, \ i\ even}^{k}\overline{L(k,i)}\otimes V_{i+1,r} \  \mbox{and} \ \ \overline{M_{r}^{odd}}=\bigoplus_{i=0, \ i\ odd}^{k}\overline{L(k,i)}\otimes V_{i+1,r}
  \end{eqnarray} for $1\leq r\leq 2k+2$, where $\overline{L(k,i)}$ for $0\leq i\leq k$ are irreducible $\sigma$-twisted modules of $L_{\hat{sl_2}}(k,0)$.
\end{prop}

\begin{proof} Since $L^{even}_{\hat{g}}(k,0)$ is $\sigma$-rational. From \cite{DLM4}, we know that the number of inequivalent irreducible $\sigma$-twisted modules of $L^{even}_{\hat{g}}(k,0)$ is precisely the number of $\sigma$-stable irreducible untwisted modules of $L^{even}_{\hat{g}}(k,0)$. Notice that $M_{r}^{even}$ and $M_{r}^{odd}$ for $1\leq r\leq 2k+2$ exhaust all the irreducible modules for $L^{even}_{\hat{g}}(k,0)$. From Theorem \ref{thm:affine sl2} and \cite{CFK}, we know that the conformal weight of $L(k,i)\otimes V_{i+1,r}$ is
\begin{eqnarray}\label{lowest1}
\frac{1}{4}\big(2(i+1)^{2}-2(i+1)r+\frac{(k+2)(r^{2}-1)}{2k+3}\big)
 \end{eqnarray}
We see that these weights are pairwise different for $1\leq r\leq 2k+2$.  From (\ref{decomposition}), we see that $M_{r}^{even}$ and $M_{r}^{odd}$ for $1\leq r\leq 2k+2$ are $\sigma$-stable irreducible modules. Thus there are totally $4k+4$ inequivalent irreducible $\sigma$-twisted modules of $L_{\hat{\g}}^{even}(k,0)$.
\end{proof}
Let $\{ h, e, f\}$
be a standard Chevalley basis of $sl_2$ with brackets $[h,e] = 2e$, $[h,f] = -2f$,
$[e,f] = h$. Set $$h^{'}=e+f, \ e^{'}=\frac{1}{2}(h-e+f),\ f^{'}=\frac{1}{2}(h+e-f).$$ Then $\{h^{'}, e^{'}, f^{'}\}$ is a $sl_2$-triple. Let $h^{''}=\frac{1}{4}h^{'}=\frac{1}{4}(e+f)$, from \cite{L2}, we know that $e^{2\pi \sqrt{-1}h^{''}(0)}$ is an automorphism of $L_{\hat{g}}(k,0)$ and
$$e^{2\pi \sqrt{-1}h^{''}(0)}(h^{'})=h^{'},\ e^{2\pi \sqrt{-1}h^{''}(0)}(e^{'})=-e^{'},\ e^{2\pi \sqrt{-1}h^{''}(0)}(f^{'})=-f^{'},$$
$$e^{2\pi \sqrt{-1}h^{''}(0)}(x+y)=-\sqrt{-1}(x+y),\ e^{2\pi \sqrt{-1}h^{''}(0)}(x-y)=\sqrt{-1}(x-y).$$
Thus $e^{2\pi \sqrt{-1}h^{''}(0)}=\sigma$.

Recall from \cite{JW1} that the top level $U^{i}=\bigoplus_{j=0}^{i}\mathbb{C}v^{i,j}$ of $L(k,i)$ for $0\leq i\leq k$ is an $(i+1)$-dimensional irreducible module for $\C h(0)\oplus \C e(0)\oplus \C f(0)\cong sl_2$. Let
\begin{eqnarray*}
\eta_{i}=\sum_{j=0}^{i}(-1)^{j}v^{i,j},
\end{eqnarray*} then $\eta_{i}$ is the lowest weight vector with weight $-i$ in $(i+1)$-dimensional irreducible module for $\C h^{'}(0)\oplus \C e^{'}(0)\oplus \C f^{'}(0)\cong sl_2$, that is, $f^{'}(0)\eta_{i}=0$ and $h^{'}(0)\eta_{i}=-i\eta_{i}$. Since $\sigma(\eta_{i})=e^{2\pi \sqrt{-1}h^{''}(0)}(\eta_{i})$,
 by a straightforward calculation, we have

\begin{lem}\label{lem:action.} For $0\leq i\leq k$, if $i\in 4\mathbb{Z},\ \sigma(\eta_{i})=\eta_{i}$. If $i\in 4\mathbb{Z}+2,\ \sigma(\eta_{i})=-\eta_{i}$. If $i\in 4\mathbb{Z}+1,\ \sigma(\eta_{i})=-\sqrt{-1}\eta_{i}$. If $i\in 4\mathbb{Z}+3,\ \sigma(\eta_{i})=\sqrt{-1}\eta_{i}$.
\end{lem}

Now we give the classification of the modules of the vertex operator algebra $L_{\hat{g}}(k,0)^{\sigma}$.

 \begin{thm}\label{thm:affinetwist'} The modules
 \begin{eqnarray}\label{mod1.}
 M_{r}^{even,+}=\bigoplus_{i=0, \ i\in 4\mathbb{Z}}^{k}L(k,i)^{+}\otimes V_{i+1,r}\bigoplus_{i=0, \ i\in 4\mathbb{Z}+2}^{k}L(k,i)^{-}\otimes V_{i+1,r}\end{eqnarray}
  \begin{eqnarray}\label{mod2.}
 M_{r}^{even,-}=\bigoplus_{i=0, \ i\in 4\mathbb{Z}}^{k}L(k,i)^{-}\otimes V_{i+1,r}\bigoplus_{i=0, \ i\in 4\mathbb{Z}+2}^{k}L(k,i)^{+}\otimes V_{i+1,r}\end{eqnarray}
 \begin{eqnarray}\label{mod3.}
 \overline{M_{r}^{even,+}}=\bigoplus_{i=0, \ i\in 4\mathbb{Z}}^{k}\overline{L(k,i)}^{+}\otimes V_{i+1,r}\bigoplus_{i=0, \ i\in 4\mathbb{Z}+2}^{k}\overline{L(k,i)}^{-}\otimes V_{i+1,r}\end{eqnarray}
 \begin{eqnarray}\label{mod4.}
 \overline{M_{r}^{even,-}}=\bigoplus_{i=0, \ i\in 4\mathbb{Z}}^{k}\overline{L(k,i)}^{-}\otimes V_{i+1,r}\bigoplus_{i=0, \ i\in 4\mathbb{Z}+2}^{k}\overline{L(k,i)}^{+}\otimes V_{i+1,r}\end{eqnarray}

\begin{eqnarray}\label{mod5.}M_{r}^{odd,1}=\bigoplus_{i=0, \ i\in 4\mathbb{Z}+1}^{k}L(k,i)^{+}\otimes V_{i+1,r}\bigoplus_{i=0, \ i\in 4\mathbb{Z}+3}^{k}L(k,i)^{-}\otimes V_{i+1,r}
  \end{eqnarray}
\begin{eqnarray}\label{mod6.}M_{r}^{odd,2}=\bigoplus_{i=0, \ i\in 4\mathbb{Z}+1}^{k}L(k,i)^{-}\otimes V_{i+1,r}\bigoplus_{i=0, \ i\in 4\mathbb{Z}+3}^{k}L(k,i)^{+}\otimes V_{i+1,r}
  \end{eqnarray}

\begin{eqnarray}\label{mod7.}\overline{M_{r}^{odd,1}}=\bigoplus_{i=0, \ i\in 4\mathbb{Z}+1}^{k}\overline{L(k,i)}^{+}\otimes V_{i+1,r}\bigoplus_{i=0, \ i\in 4\mathbb{Z}+3}^{k}\overline{L(k,i)}^{-}\otimes V_{i+1,r}
  \end{eqnarray}

 \begin{eqnarray}\label{mod8.}\overline{M_{r}^{odd,2}}=\bigoplus_{i=0, \ i\in 4\mathbb{Z}+1}^{k}\overline{L(k,i)}^{-}\otimes V_{i+1,r}\bigoplus_{i=0, \ i\in 4\mathbb{Z}+3}^{k}\overline{L(k,i)}^{+}\otimes V_{i+1,r}
  \end{eqnarray}

  for $1\leq r\leq 2k+2$ form a complete list of inequivalent simple modules of the vertex operator algebra $L_{\hat{g}}(k,0)^{\sigma}$.

 \end{thm}

\begin{proof} From \cite{JW2}, we know that
for $i\neq 0$, $L(k,i)^{+}$ is an $L_{\hat{sl_2}}(k,0)^{\sigma}$-module generated by $\eta_{i}$, i.e., $L(k,i)^{+}=L_{\hat{sl_2}}(k,0)^{\sigma}.\eta_{i}$,
$L(k,i)^{-}$ is an $L_{\hat{sl_2}}(k,0)^{\sigma}$-module generated by $e^{'}(0)\eta_{i}$, i.e., $L(k,i)^{-}=L_{\hat{sl_2}}(k,0)^{\sigma}.(e^{'}(0)\eta_{i})$.
And $L(k,0)^{+}=L_{\hat{sl_2}}(k,0)^{\sigma}.\mathbf{1}$, $L(k,0)^{-}=L_{\hat{sl_2}}(k,0)^{\sigma}.(e^{'}(-1)\mathbf{1})$. $\overline{L(k,i)}^{+}$ is an $L_{\hat{sl_2}}(k,0)^{\sigma}$-module generated (in twisted module $\overline{L(k,i)}$) by $\eta_{i}$, we still denoted it by $\overline{L(k,i)}^{+}=L_{\hat{sl_2}}(k,0)^{\sigma}.\eta_{i}$, and $\overline{L(k,i)}^{-}=L_{\hat{sl_2}}(k,0)^{\sigma}.((e-f)_{-\frac{1}{2}}\eta_{i})$. Then by \cite{DM1} and
Lemma \ref{lem:action.}, we obtain (\ref{mod1.})-(\ref{mod8.}). Where we notice that $$M_{r}^{even,\pm}=\{v\in M_{r}^{even}|\sigma(v)=\pm v\},\ \ \overline{M_{r}^{even,\pm}}=\{v\in \overline{M_{r}^{even}}|\sigma(v)=\pm v\},$$ $$M_{r}^{odd,1}=\{v\in M_{r}^{odd}|\sigma(v)=-\sqrt{-1}v\}, \ M_{r}^{odd,2}=\{v\in M_{r}^{odd}|\sigma(v)=\sqrt{-1}v\},$$
$$\overline{M_{r}^{odd,1}}=\{v\in \overline{M_{r}^{odd}}|\sigma(v)=-\sqrt{-1}v\}, \ \overline{M_{r}^{odd,2}}=\{v\in \overline{M_{r}^{odd}}|\sigma(v)=\sqrt{-1}v\}.$$

\end{proof}

 \begin{rmk}\label{rmk:affine orbifold} With the notations in Theorem \ref{thm:affinetwist'}, we call $M_{r}^{even,\pm}, M_{r}^{odd,1}, M_{r}^{odd,2}$ untwisted type modules and  $\overline{M_{r}^{even,\pm}}, \overline{M_{r}^{odd,1}}, \overline{M_{r}^{odd,2}}$ twisted type modules respectively.

\end{rmk}

 \begin{thm}\label{thm:contragredient aff}
 For $1\leq r\leq 2k+2$.

  (1) $M_{r}^{even,\pm}$ are self-dual, and $(M_{r}^{odd,1})^{'}\cong M_{r}^{odd,2},\ (M_{r}^{odd,2})^{'}\cong M_{r}^{odd,1}.$

(2) If $k\in 4\mathbb{Z}$, $(\overline{M_{r}^{even,\pm}})^{'}=\overline{M_{2k+3-r}^{even,\pm}}.$
If $k\in 4\mathbb{Z}+2$, $(\overline{M_{r}^{even,+}})^{'}=\overline{M_{2k+3-r}^{even,-}}.$\\

\ \ \ \  If $k\in 4\mathbb{Z}+1$, $(\overline{M_{r}^{even,+}})^{'}=\overline{M_{2k+3-r}^{odd,1}}.$
 If $k\in 4\mathbb{Z}+3$, $(\overline{M_{r}^{even,+}})^{'}=\overline{M_{2k+3-r}^{odd,2}}.$\\

\ \ \ \  If $k\in 4\mathbb{Z}$, $(\overline{M_{r}^{odd,1}})^{'}=\overline{M_{2k+3-r}^{odd,2}}.$\\

 \ \ \ \  If $k\in 4\mathbb{Z}+2$, $(\overline{M_{r}^{odd,1}})^{'}=\overline{M_{2k+3-r}^{odd,1}},$ $(\overline{M_{r}^{odd,2}})^{'}=\overline{M_{2k+3-r}^{odd,2}}.$\\

\ \ \ \   If $k\in 4\mathbb{Z}+1$, $(\overline{M_{r}^{odd,2}})^{'}=\overline{M_{2k+3-r}^{even,-}}.$
 If $k\in 4\mathbb{Z}+3$, $(\overline{M_{r}^{odd,1}})^{'}=\overline{M_{2k+3-r}^{even,-}}.$
\end{thm}
\begin{proof} We prove the case that if $k\in 4\mathbb{Z}+1$, $(\overline{M_{r}^{even,+}})^{'}=\overline{M_{2k+3-r}^{odd,1}}.$ Other cases are similar to prove. Since the modules $V_{i,r}$ for $0\leq i\leq k,\ 1\leq r\leq 2k+2$ are self-dual, if $k\in 4\mathbb{Z}+1$, we have

\begin{eqnarray*}
\begin{split}(\overline{M_{r}^{even,+}})^{'} &= (\bigoplus\limits_{\tiny{\begin{split}i=0 \ \\ i\in 4\mathbb{Z}\end{split}}}^{k}(\overline{L(k,i)}^{+}\otimes V_{i+1,r})^{'}\bigoplus\limits_{\tiny{\begin{split}i=0 \ \ \ \ \\ i\in 4\mathbb{Z}+2\end{split}}}^{k}(\overline{L(k,i)}^{-}\otimes V_{i+1,r}))^{'}\\
&=\bigoplus\limits_{\tiny{\begin{split}i=0 \ \\  i\in 4\mathbb{Z}\end{split}}}^{k}(\overline{L(k,i)}^{+})^{'}\otimes V_{i+1,r}\bigoplus\limits_{\tiny{\begin{split} i=0 \ \ \ \\   i\in 4\mathbb{Z}+2\end{split}}}^{k}(\overline{L(k,i)}^{-})^{'}\otimes V_{i+1,r}\\
&=\bigoplus\limits_{\tiny{\begin{split}i=0 \ \\  i\in 4\mathbb{Z}\end{split}}}^{k}\overline{L(k,k-i)}^{+}\otimes V_{i+1,r}\bigoplus\limits_{\tiny{\begin{split} i=0 \ \ \ \\   i\in 4\mathbb{Z}+2\end{split}}}^{k}\overline{L(k,k-i)}^{-}\otimes V_{i+1,r}\\
&=\bigoplus\limits_{\tiny{\begin{split}i=0 \ \\  i\in 4\mathbb{Z}\end{split}}}^{k}\overline{L(k,k-i)}^{+}\otimes V_{k+1-i,2k+3-r}\bigoplus\limits_{\tiny{\begin{split} i=0 \ \ \ \\   i\in 4\mathbb{Z}+2\end{split}}}^{k}\overline{L(k,k-i)}^{-}\otimes V_{k+1-i,2k+3-r},
\end{split}
\end{eqnarray*}
where in the last equality we used the fact that $V_{i,r}\cong V_{k+2-i,2k+3-r}$ and in the third equality we used the second result of Theorem 3.25 in \cite{JW2}. Moreover, we notice that if $k\in 4\mathbb{Z}+1$ and $i\in 4\mathbb{Z}$, we have $k-i\in 4\mathbb{Z}+1$. If $k\in 4\mathbb{Z}+1$ and $i\in 4\mathbb{Z}+2$, we have $k-i\in 4\mathbb{Z}+3$. Thus we obtain $(\overline{M_{r}^{even,+}})^{'}=\overline{M_{2k+3-r}^{odd,1}}.$

\end{proof}

\section{Fusion rules for the orbifold of affine vertex operator superalgebra $L_{\hat{g}}(k,0)$
}\label{Sect: fusion product}\def\theequation{5.\arabic{equation}}
\setcounter{equation}{0}

In this section, we determine the fusion rules for $L_{\hat{g}}(k,0)^{\sigma}$. For simplicity, we denote untwisted type modules $M_{r}^{odd,1}$ by $M_{r}^{odd,+}$ and $M_{r}^{odd,2}$ by $M_{r}^{odd,-}$. We denote twisted type modules $\overline{M_{r}^{odd,1}}$ by $\overline{M_{r}^{odd,+}}$ and $\overline{M_{r}^{odd,2}}$ by $\overline{M_{r}^{odd,-}}$. And for irreducible $L_{\hat{g}}(k,0)^{\sigma}$-modules $W^{1}$ and $W^{2}$, we use $W^{1}\boxtimes W^{2}$ to denote the fusion product $W^{1}\boxtimes_{L_{\hat{g}}(k,0)^{\sigma}}W^{2}$.\\

Now we  recall from \cite{FHL} the notions of intertwining operators and fusion rules.

\begin{definition} Let $\left(V,Y,1,\omega\right)$ be a vertex operator algebra and
let $(W^{1},\ Y^{1}),\ (W^{2},\ Y^{2})$ and $(W^{3},\ Y^{3})$ be
$V$-modules. An \emph{intertwining operator} of type $\left(\begin{array}{c}
W^{3}\\
W^{1\ }W^{2}
\end{array}\right)$ is a linear map
\[
I(\cdot,\ z):\ W^{1}\to\text{\ensuremath{\mbox{Hom}(W^{2},\ W^{3})\{z\}}}
\]

\[
u\to I(u,\ z)=\sum_{n\in\mathbb{Q}}u_{n}z^{-n-1}
\]
 satisfying:

(1) for any $u\in W^{1}$ and $v\in W^{2}$, $u_{n}v=0$ for $n$
sufficiently large;

(2) $I(L(-1)v,\ z)=\frac{d}{dz}I(v,\ z)$;

(3) (Jacobi identity) for any $u\in V,\ v\in W^{1}$

\[
z_{0}^{-1}\delta\left(\frac{z_{1}-z_{2}}{z_{0}}\right)Y^{3}(u,\ z_{1})I(v,\ z_{2})-z_{0}^{-1}\delta\left(\frac{-z_{2}+z_{1}}{z_{0}}\right)I(v,\ z_{2})Y^{2}(u,\ z_{1})
\]
\[
=z_{2}^{-1}\left(\frac{z_{1}-z_{0}}{z_{2}}\right)I(Y^{1}(u,\ z_{0})v,\ z_{2}).
\]

The space of all intertwining operators of type $\left(\begin{array}{c}
W^{3}\\
W^{1}\ W^{2}
\end{array}\right)$ is denoted by
$$I_{V}\left(\begin{array}{c}
W^{3}\\
W^{1}\ W^{2}
\end{array}\right).$$ Let $N_{W^{1},\ W^{2}}^{W^{3}}=\dim I_{V}\left(\begin{array}{c}
W^{3}\\
W^{1}\ W^{2}
\end{array}\right)$. These integers $N_{W^{1},\ W^{2}}^{W^{3}}$ are usually called the
\emph{fusion rules}. \end{definition}




\begin{definition} Let $V$ be a vertex operator algebra, and $W^{1},$
$W^{2}$ be two $V$-modules. A module $(W,I)$, where $I\in I_{V}\left(\begin{array}{c}
\ \ W\ \\
W^{1}\ \ W^{2}
\end{array}\right),$ is called a \emph{tensor product} (or fusion product) of $W^{1}$
and $W^{2}$ if for any $V$-module $M$ and $\mathcal{Y}\in I_{V}\left(\begin{array}{c}
\ \ M\ \\
W^{1}\ \ W^{2}
\end{array}\right),$ there is a unique $V$-module homomorphism $f:W\rightarrow M,$ such
that $\mathcal{Y}=f\circ I.$ As usual, we denote $(W,I)$ by $W^{1}\boxtimes_{V}W^{2}.$
\end{definition}

\begin{rmk}
It is well known that if $V$ is rational, then for any two irreducible
$V$-modules $W^{1}$ and $W^{2},$ the fusion product $W^{1}\boxtimes_{V}W^{2}$ exists and
$$
W^{1}\boxtimes_{V}W^{2}=\sum_{W}N_{W^{1},\ W^{2}}^{W}W,
$$
 where $W$ runs over the set of equivalence classes of irreducible
$V$-modules.
\end{rmk}
Fusion rules have the following symmetric property \cite{FHL}.

\begin{prop}\label{fusionsymm.}
Let $W^{i} (i=1,2,3)$ be $V$-modules. Then
$$N_{W^{1},W^{2}}^{W^{3}}=N_{W^{2},W^{1}}^{W^{3}}, \ N_{W^{1},W^{2}}^{W^{3}}=N_{W^{1},(W^{3})^{'}}^{(W^{2})^{'}}.$$
\end{prop}

 We recall from \cite{JW2} the fusion rules for $L_{\hat{sl_2}}(k,0)^{\sigma}$ for later use.
For $0\leq i\leq k, \ 0\leq j\leq k,\ 0\leq l\leq k$ such that $i+j+l\in 2\mathbb{Z}$, we define

\[\mbox{sign}(i,j,l)^{+}=\begin{cases}+,\ & \mbox{if} \ i+j-l\in 4{\mathbb{Z}},\cr
-,\ &\mbox{if} \  i+j-l\notin 4{\mathbb{Z}},\end{cases}\]

and

\[\mbox{sign}(i,j,l)^{-}=\begin{cases}-,\ & \mbox{if} \ i+j-l\in 4{\mathbb{Z}},\cr
+,\ &\mbox{if} \  i+j-l\notin 4{\mathbb{Z}}.\end{cases}\]

\begin{thm}\cite{JW2}\label{fusion.aff.} The fusion rules for the ${\mathbb{Z}}_{2}$-orbifold affine vertex operator algebra $L(k,0)^{\sigma}$ are as follows:
\begin{eqnarray}\label{fusion.untwist1..}
L(k,i)^{+}\boxtimes L(k,j)^{\pm}=\sum\limits_{\tiny{\begin{split}|i-j|\leq l\leq i+j \\  i+j+l\in 2\mathbb{Z} \ \ \ \\ i+j+l\leq 2k\ \ \ \end{split}}} L(k,l)^{\mbox{sign}(i,j,l)^{\pm}},
\end{eqnarray}

\begin{eqnarray}\label{fusion.untwist2}
L(k,i)^{-}\boxtimes L(k,j)^{\pm}=\sum\limits_{\tiny{\begin{split}|i-j|\leq l\leq i+j \\  i+j+l\in 2\mathbb{Z} \ \ \ \\ i+j+l\leq 2k\ \ \ \end{split}}}  L(k,l)^{\mbox{sign}(i,j,l)^{\mp}},
\end{eqnarray}

\begin{eqnarray}\label{fusion.twist1}
L(k,i)^{+}\boxtimes \overline{L(k,j)}^{\pm}=\sum\limits_{\tiny{\begin{split}|i-j|\leq l\leq i+j \\  i+j+l\in 2\mathbb{Z} \ \ \ \\ i+j+l\leq 2k\ \ \ \end{split}}} \overline{L(k,l)}^{\mbox{sign}(i,j,l)^{\pm}},
\end{eqnarray}

\begin{eqnarray}\label{fusion.twist2}
L(k,i)^{-}\boxtimes \overline{L(k,j)}^{\pm}=\sum\limits_{\tiny{\begin{split}|i-j|\leq l\leq i+j \\  i+j+l\in 2\mathbb{Z} \ \ \ \\ i+j+l\leq 2k\ \ \ \end{split}}}  \overline{L(k,l)}^{\mbox{sign}(i,j,l)^{\mp}}.
\end{eqnarray}
\end{thm}

 We recall from \cite{CL1} and \cite{CFK} that the coset $$Com(L_{\hat{sl_2}}(k,0), L_{\hat{\g}}(k,0))=L^{Vir}(c_{2k+3,k+2},0),$$ where $L^{Vir}(c_{2k+3,k+2},0)$ is the minimal Virasoro vertex operator algebra with central charge $c_{2k+3,k+2}=1-\frac{6(k+1)^{2}}{(2k+3)(k+2)}$. And the simple modules of the Virasoro vertex operator algebra $L^{Vir}(c_{2k+3,k+2},0)$ are denoted by $V_{r,s}$ for $1\leq r\leq k+1$ and $1\leq s\leq 2k+2$. Let \[N_{r,r^{'}}^{u,r^{''}}=\begin{cases}1,\ & \mbox{if} \ |r-r^{'}|+1\leq r^{''}\leq \mbox{min}\{r+r^{'}-1,2u-r-r^{'}\}, \ r+r^{'}+r^{''} odd,\cr
0,\ &\mbox{otherwise}, \end{cases}\] then the fusion rules of the Virasoro vertex operator algebra $L^{Vir}(c_{2k+3,k+2},0)$ can be expressed as follows \cite{W}:

\begin{eqnarray*}
V_{r,s}\boxtimes V_{r^{'},s^{'}}=\bigoplus_{r^{''}=1}^{k+1}\bigoplus_{s^{''}=1}^{2k+2}N_{r,r^{'}}^{k+2,r^{''}}N_{s,s^{'}}^{2k+3,s^{''}}V_{r^{''},s^{''}}.
\end{eqnarray*}

Now we can state our main result in this section.

\begin{thm}\label{aff-fusion-untwist} The fusion rules for the irreducible modules of the orbifold vertex operator algebra $L_{\hat{g}}(k,0)^{\sigma}$ are as follows:

\begin{eqnarray}\label{fusion.1..}
M_{r}^{even,+}\boxtimes M_{r^{'}}^{even,\pm}=\bigoplus_{r^{''}=1}^{2k+2} N_{r,r^{'}}^{r^{''}}M_{r^{''}}^{even,\pm}
\end{eqnarray}

\begin{eqnarray}\label{fusion.2..}
M_{r}^{even,-}\boxtimes M_{r^{'}}^{even,\pm}=\bigoplus_{r^{''}=1}^{2k+2} N_{r,r^{'}}^{r^{''}}M_{r^{''}}^{even,\mp}
\end{eqnarray}

\begin{eqnarray}\label{fusion.3..}
M_{r}^{even,+}\boxtimes \overline{M_{r^{'}}^{even,\pm}}=\bigoplus_{r^{''}=1}^{2k+2} N_{r,r^{'}}^{r^{''}}\overline{M_{r^{''}}^{even,\pm}}
\end{eqnarray}

\begin{eqnarray}\label{fusion.4..}
M_{r}^{even,-}\boxtimes \overline{M_{r^{'}}^{even,\pm}}=\bigoplus_{r^{''}=1}^{2k+2} N_{r,r^{'}}^{r^{''}}\overline{M_{r^{''}}^{even,\mp}}
\end{eqnarray}

\begin{eqnarray}\label{fusion.5..}
M_{r}^{odd,+}\boxtimes M_{r^{'}}^{odd,\pm}=\bigoplus_{r^{''}=1}^{2k+2} N_{r,r^{'}}^{r^{''}}M_{r^{''}}^{even,\mp}
\end{eqnarray}


\begin{eqnarray}\label{fusion.7..}
M_{r}^{odd,-}\boxtimes M_{r^{'}}^{odd,\pm}=\bigoplus_{r^{''}=1}^{2k+2} N_{r,r^{'}}^{r^{''}}M_{r^{''}}^{even,\pm}
\end{eqnarray}

\begin{eqnarray}\label{fusion.8..}
M_{r}^{odd,+}\boxtimes \overline{M_{r^{'}}^{odd,\pm}}=\bigoplus_{r^{''}=1}^{2k+2} N_{r,r^{'}}^{r^{''}}\overline{M_{r^{''}}^{even,\mp}}
\end{eqnarray}


\begin{eqnarray}\label{fusion.10..}
M_{r}^{odd,-}\boxtimes \overline{M_{r^{'}}^{odd,\pm}}=\bigoplus_{r^{''}=1}^{2k+2} N_{r,r^{'}}^{r^{''}}\overline{M_{r^{''}}^{even,\pm}}
\end{eqnarray}


\begin{eqnarray}\label{fusion.12..}
M_{r}^{odd,+}\boxtimes M_{r^{'}}^{even,\pm}=\bigoplus_{r^{''}=1}^{2k+2} N_{r,r^{'}}^{r^{''}}M_{r^{''}}^{odd,\pm}
\end{eqnarray}


\begin{eqnarray}\label{fusion.14..}
M_{r}^{odd,-}\boxtimes M_{r^{'}}^{even,\pm}=\bigoplus_{r^{''}=1}^{2k+2} N_{r,r^{'}}^{r^{''}}M_{r^{''}}^{odd,\mp}
\end{eqnarray}


\begin{eqnarray}\label{fusion.16..}
M_{r}^{odd,+}\boxtimes \overline{M_{r^{'}}^{even,\pm}}=\bigoplus_{r^{''}=1}^{2k+2} N_{r,r^{'}}^{r^{''}}\overline{M_{r^{''}}^{odd,\pm}}
\end{eqnarray}


\begin{eqnarray}\label{fusion.18..}
M_{r}^{odd,-}\boxtimes \overline{M_{r^{'}}^{even,\pm}}=\bigoplus_{r^{''}=1}^{2k+2} N_{r,r^{'}}^{r^{''}}\overline{M_{r^{''}}^{odd,\mp}}
\end{eqnarray}


where $N_{r,r^{'}}^{r^{''}}=N_{r,r^{'}}^{2k+3, r^{''}}$, $1\leq r,r^{'}\leq 2k+2$.

\end{thm}
\begin{proof} We prove (\ref{fusion.1..}) and (\ref{fusion.5..}), others are similar to prove. First we notice  that $M_{1}^{even}=L^{even}_{\hat{g}}(k,0)$, thus $L^{even}_{\hat{g}}(k,0)^{\sigma}=M_{1}^{even,+}$. We denote $M_{1}^{even,+}$ by $L^{even,+}$, then we have $L^{even}_{\hat{g}}(k,0)=L^{even,+}\oplus L^{even,-}$.
Since we have

\begin{equation}\label{fusion.1'..}
\begin{split}
 M_{r}^{even,+}\boxtimes M_{r^{'}}^{even,+} &= L^{even,+}\boxtimes (L(k,0)^{+}\otimes V_{1,r})\boxtimes L^{even,+}\boxtimes (L(k,0)^{+}\otimes V_{1,r^{'}})
\\
&=L^{even,+}\boxtimes \big((L(k,0)^{+}\otimes V_{1,r})\boxtimes (L(k,0)^{+}\otimes V_{1,r^{'}})\big)\\
&=L^{even,+}\boxtimes (\bigoplus_{r^{''}=1}^{2k+2}N_{r,r^{'}}^{2k+3,r^{''}}(L(k,0)^{+}\otimes V_{1,r^{''}}))\\
&=\bigoplus_{r^{''}=1}^{2k+2}N_{r,r^{'}}^{r^{''}} M_{r^{''}}^{even,+},
\end{split}
 \end{equation}
where we denote $N_{r,r^{'}}^{2k+3,r^{''}}$ by $N_{r,r^{'}}^{r^{''}}$, and in the third equality we used (\ref{fusion.untwist1..}) in Theorem \ref{fusion.aff.}. Thus we prove  (\ref{fusion.1..}). From the definition of $M_{r}^{even,\pm}$ and $ M_{r}^{odd,\pm}$, we have
$$M_{r}^{even,+}=L^{even,+}\boxtimes M_{r}^{even,+},\ M_{r}^{even,-}=L^{even,-}\boxtimes \  M_{r}^{even,+},$$
$$M_{r}^{odd,+}=L^{odd,+}\boxtimes \  M_{r}^{even,+},\ M_{r}^{odd,-}=L^{odd,-}\boxtimes M_{r}^{even,+},$$
where $L^{odd,\pm}=M_{1}^{odd,\pm}$.
Thus we have
\begin{equation*}
\begin{split}
 M_{r}^{odd,+}\boxtimes M_{r^{'}}^{odd,+} &=L^{odd,+}\boxtimes M_{r}^{even,+} \boxtimes L^{odd,+}\boxtimes M_{r^{'}}^{even,+}
\\
&=L^{odd,+}\boxtimes L^{odd,+} \boxtimes M_{r}^{even,+}\boxtimes M_{r^{'}}^{even,+}\\
&=L^{even,-}\boxtimes (\bigoplus_{r^{''}=1}^{2k+2}N_{r,r^{'}}^{r^{''}} M_{r^{''}}^{even,+})\\
&=\bigoplus_{r^{''}=1}^{2k+2}N_{r,r^{'}}^{r^{''}} M_{r^{''}}^{even,-},
\end{split}
 \end{equation*}
  where in the third equality, we used (\ref{fusion.1'..}) proved above. Thus we obtain (\ref{fusion.5..}).

 For  (\ref{fusion.3..}), (\ref{fusion.4..}), (\ref{fusion.8..}), (\ref{fusion.10..}), (\ref{fusion.16..}), (\ref{fusion.18..}), notice that $$\overline{M_{r}^{even,+}}=L^{even,+}\boxtimes \overline{M_{r}^{even,+}},\ \overline{M_{r}^{even,-}}=L^{even,-}\boxtimes \overline{M_{r}^{even,+}},$$
$$\overline{M_{r}^{odd,+}}=L^{odd,+}\boxtimes\overline{ M_{r}^{even,+}},\ \overline{M_{r}^{odd,-}}=L^{odd,-}\boxtimes \overline{M_{r}^{even,+}}.$$ By a similar argument, we can obtain the results.
\end{proof}






\end{document}